\documentclass[12pt,reqno]{amsart}
\usepackage{amssymb}
\usepackage{amscd}
\usepackage{hyperref}

\usepackage{amsxtra}

\usepackage{hyperref}

\usepackage{esint}

\usepackage[shortlabels]{enumitem}
\usepackage{cleveref}

\crefrangeformat{enumi}{items #3#1#4--#5#2#6}

\theoremstyle{plain}

\newtheorem{thm}{Theorem}[section]
\newtheorem{lem}[thm]{Lemma}
\newtheorem{cor}[thm]{Corollary}
\newtheorem{prop}[thm]{Proposition}

\theoremstyle{definition}

\theoremstyle{remark}

\newtheorem{rem}[thm]{Remark}

\crefname{thm}{theorem}{theorems}
\crefname{lem}{lemma}{lemmas}
\crefname{cor}{corollary}{corollaries}
\crefname{prop}{proposition}{propositions}
\crefname{mainthm}{theorem}{theorems}
\crefname{maincor}{corollary}{corollaries}

\crefname{defn}{definition}{definitions}
\crefname{conj}{conjecture}{conjectures}
\crefname{example}{example}{examples}
\crefname{exercise}{exercise}{exercises}
\crefname{prob}{problem}{problems}
\crefname{quest}{question}{questions}

\crefname{rem}{remark}{remarks}
\crefname{claim}{claim}{claims}
\crefname{axiom}{axiom}{axioms}
\crefname{hyp}{hypothesis}{hypotheses}
\crefname{notation}{notation}{notations}
\crefname{case}{case}{cases}

\numberwithin{equation}{section}

\usepackage{color}
\definecolor{darkgreen}{cmyk}{1,0,1,.2}
\definecolor{m}{rgb}{1,0.1,1}

\newdimen\theight
\def\TeXref#1{%
             \leavevmode\vadjust{\setbox0=\hbox{{\tt
                     \quad\quad  {\small \textrm #1}}}%
             \theight=\ht0
             \advance\theight by \lineskip
             \kern -\theight \vbox to
             \theight{\rightline{\rlap{\box0}}%
             \vss}%
             }}%

\begin{document}
\vskip .5cm
\title[ Harmonic Maps on Weighted Foliations]
{Harmonic Maps  on Weighted Riemannian Foliations}

\author[X. S. Fu]{Xueshan Fu}
\address{Department of Mathematics\\
         Shenyang University of Technology\\
         Shenyang 110870\\
         China}
\email{xsfu@sut.edu.cn}

\author[J. H. Qian]{Jinhua Qian}
\address{Department of Mathematics\\
Northeastern University\\
Shenyang 110819\\
China}
\email{qianjinhua@mail.neu.edu.cn}

\author[S. D. Jung]{Seoung Dal Jung}
\address{Department of Mathematics\\
         Jeju National University\\
         Jeju 63243\\
         Republic of Korea}
\email{sdjung@jejunu.ac.kr}

\thanks{This paper is supported by Research Ability Cultivation Fund for Young Teachers of Shenyang University of Technology (QNPY202209-24), NSFC (No. 11801065) and  NRF-2022R1A2C1003278.}

\subjclass[2010]{53C12; 53C21; 58J50.}
\keywords{Weighted foliation,  Transversal Bakry-\'Emery Ricci tensor, Transversally  $f$-harmonic map, $(\mathcal F,\mathcal F')_f$-harmonic map, Liouville type theorem.}

\begin{abstract}
On foliations, there are two kinds of harmonic maps, that is, transversally harmonic map and $(\mathcal F,\mathcal F')$-harmonic map which are equivalent  when the foliation is minimal.  In this paper, we study  transversally $f$-harmonic and $(\mathcal F,\mathcal F')_f$-harmonic maps  on weighted Riemannian foliations.  In particular, we study the Liouville type theorem for such $f$-harmonic maps on weighted Riemannian foliations. Moreover, we investigate the relations between transversal stress  energy tensors and tranversally $f$-harmonic maps.
\end{abstract}
\maketitle

\section{Introduction}
 Let $(M,g,\mathcal F)$ and $(M',g',\mathcal F')$ be foliated Riemannian manifolds and $\phi:M\to M'$ be a smooth foliated map (i.e., $\phi$ is a smooth leaf-preserving map). 
 The smooth foliated map $\phi$ is said to be {\it transversally harmonic} if  $\phi$ is a solution of the Eular-Largrange equation $\tau_{b}(\phi)=0$, where  $\tau_b(\phi)$ is the transversal tension field of $\phi$ (see Section 3). 
 Transversally harmonic maps on foliated Riemannian manifolds have been studied by many authors \cite{CZ,FJ,JU3,JJ1,JJ2,KW1,KW2,OSU}.  However,   a transversally harmonic map is not a critical point of the transversal energy functional \cite{JJ1}
\begin{align*}
E_{B}(\phi)=\frac{1}{2}\int_{M} | d_T \phi|^2\mu_{M}.
\end{align*}
In 2013, S. Dragomir and A. Tommasoli \cite{DT} defined a new harmonic map, called {\it $(\mathcal F,\mathcal F')$-harmonic map}, which is a critical point of the transversal energy functional $E_{B}(\phi)$. Two definitions are equivalent when $\mathcal F$ is minimal.  As a generalization of harmonic map,  Y. Chiang and R. Wolak \cite{CW}  defined   transvesally $f$-harmonic map for  a basic function  $f$ on $M$.  The map $\phi$ is said to be  {\it transversally $f$-harmonic} if $\phi$  is a solution of the Eular-Largrange equation  $\tau_{b,f}(\phi)=0$, where  $\tau_{b,f}(\phi)$ is the {\it transversal $f$-tension field} of $\phi$ defined by $\tau_{b,f}(\phi) ={\rm tr}_Q(\nabla_{\rm tr}(e^{-f} d_T\phi)).$  From the first variation formula (Theorem 3.4),   the transvesally $f$-harmonic map is not a critical point of the transversal $f$-energy functional 
\begin{align*}
E_{B,f}(\phi) = \frac12\int_M e^{-f}|d_T\phi|^2 \mu_M.
\end{align*}
Similarly, the map $\phi$ is said to be {\it $(\mathcal F,\mathcal F')_f $-harmonic} if  $\phi$ is a critical point of the transversal $f$-energy  functional $E_{B,f}(\phi)$.
If $f$ is  constant, then  transversally $f$-harmonic (resp. $(\mathcal F,\mathcal F')_f$-harmonic) map is just   transversally harmonic (resp. $(\mathcal F,\mathcal F')$-harmonic) map.  
Originally, $f$-harmonic maps on Riemannian manifolds  were studied by  A. Lichnerowicz in 1969   \cite{LI}, later by J.Eells and L. Lemaire  in 1977 \cite{EL}. 

In this article, we study  $f$-harmonic maps (that is, transversally $f$-harmonic map and $(\mathcal F,\mathcal F')_f$-harmonic map) on Riemannian foliations, in particular,  on weighted Riemannian foliations.  A {\it weighted Riemannian foliation} $(M,g,\mathcal F,e^{-f}\nu)$ is a Riemannian foliation endowed with a  weighted transversal volume  form $e^{-f}\nu$ and some basic function $f$, where $\nu$ is the transversal volume form of $\mathcal F$. 
The geometry of a weighted manifold (or a smooth metric measure space) was developed by D. Bakry and M. \'Emery \cite{BE} and studied by many authors \cite{LO,LV,ST,ST2,VI,WW}.  Also,  the geometry of weighted manifolds is closely related with that of self-shrinkers and gradient Ricci solitons. An important geometric tool  is  the Bakry-\'Emery Ricci tensor, which was first introduced by A. Lichnerowicz \cite{LI1}.

On a weighted Riemannian foliation $(M,g,\mathcal F,e^{-f}\nu)$, we define the {\it transversal Bakry-\'Emery Ricci tensor} ${\rm Ric}_f^Q$ on $M$  by
\begin{align*}
 {\rm Ric}_f^Q = {\rm Ric}^Q +{\rm Hess}_T f,
 \end{align*}
 where ${\rm Ric}^Q$ is the transversal Ricci tensor  and ${\rm Hess}_T f$ is the transversal Hessian \cite{JU3} of $\mathcal F$.  The transversal Bakry-\'Emery-Ricci curvature is very useful  in the study of a weighted foliation.    
 The  transversal Bakry-\'Emery-Ricci curvature  ${\rm Ric}_f^Q$ is related to the transversal Ricci solitions, which are  special solutions of the transversal Ricci flow \cite{LIN}.  Also,  it is related to the study of the transversally $f$-harmonic maps on a weighted Riemannian foliation. That is,  we have the following theorem.
 
\begin{thm} (cf. Theorem 3.7)
Let $(M,g,\mathcal F,e^{-f}\nu)$ be a weighted Riemannian foliation on a closed manifold $M$ with ${\rm Ric}_{f}^Q\geq 0$ and $(M',g',\mathcal F')$ be a Riemannian foliation with $K^{Q'}\leq 0$. Then a transversally $f$-harmonic (or $(\mathcal F,\mathcal F')_f$-harmonic)  map $\phi:M\to M'$ is always transversally totally geodesic. 
In addition, 

(1) if ${\rm Ric}_f^Q>0$ at some point, then $\phi$ is transversally constant.

(2) if  $K^{Q'}<0$, then $\phi$ is transversally constant or $\phi(M)$  is a transversally geodesic closed curve.

\end{thm}
Now we study Theorem 1.1 on a complete  manifold,  that is, the Liouville type theorems for $(\mathcal F,\mathcal F')_f$ and  transversally $f$-harmonic map, respectively. The Liouville type theorem  has been studied by many researchers \cite{JU1,EN,SY,YA} on Riemannian manifolds and \cite{FJ,JJ1,JJ2,OSU} on foliations. Specially, see \cite{RV,WX} for  $f$-harmonic maps on weighted Riemannian manifolds.  Let $f_\kappa$ be a basic function satisfying  $g_Q(\kappa_B^\sharp,\nabla f) \leq 0$, where $\kappa_B$ is the basic part of the mean curvature form $\kappa$ of $\mathcal F$.  Then we have the following theorem.

\begin{thm}   (cf. Theorem 4.2)  Let $(M,g,\mathcal  F,e^{-f_\kappa}\nu)$ be a weighted Riemannian foliation on a complete manifold whose all leaves are compact and the mean curvature form is bounded.  Let $(M',g',\mathcal F')$ be a foliated Riemannian manifold with  non-positive transversal sectional curvature $K^{Q'}$.   Let  $\phi:M\to N$ be a transversally $f_\kappa$-harmonic (or $(\mathcal F,\mathcal F')_{f_\kappa}$-harmonic) map 
with $E_{B,f_\kappa}(\phi)<\infty$.  Then

(1) if ${\rm Ric}_{f_\kappa}^Q \geq 0$, then $\phi$ is transversally totally geodesic;

(2) if ${\rm Ric}_{f_\kappa}^Q\geq 0$ and either ${\rm Ric}^Q_{f_\kappa}>0$ at some point  or $\int_M e^{-f_\kappa}\mu_M= \infty$, then $\phi$ is transvesally constant.
\end{thm}
 Lastly, we study the transversal stress energy tensors, which is appeared  in the variation formula of the transversal metric. It is well known \cite{JU3} that  a transversally harmonic map satisfies the transverse conservation law, but a transversally $f$-harmonic map does not satisfy the transverse $f$-conservation law (Remark 5.5).  So we define  the transversal $F$-stress energy tensor and show that a transversally $f$-harmonic map satisfies the transverse $F$-consevation law (cf. Proposition 5.11) for special function $f$ related to $F$, where $F:[0,\infty)\to [0,\infty)$ is  a $C^2$-function.

\section{Preliminaries}

Let $(M,g,\mathcal F)$ be a foliated Riemannian
manifold of dimension $n$ with a foliation $\mathcal F$ of codimension $q (=n-p)$ and a bundle-like metric $g$ with respect to $\mathcal F$ \cite{Molino,Tond}.  Let $Q=TM/T\mathcal F$ be the normal  bundle of $\mathcal F$, where  $T\mathcal F$ is the tangent bundle of $\mathcal F$. Let $g_Q$ be the induced metric by $g$ on $Q$, that is, $g_Q = \sigma^*(g|_{T\mathcal F^\perp})$, where $\sigma:Q\to T\mathcal F^\perp$ is the canonical bundle isomorphism. So we consider $Q\cong T\mathcal F^\perp$.  Then $g_Q$ is the holonomy invariant metric on $Q$, meaning that  $L_Xg_Q=0$ for $X\in T\mathcal F$, where
$L_X$ is the transverse Lie derivative with respect to $X$.  Let $\nabla^Q$  be the transverse Levi-Civita
connection on the normal bundle $Q$ \cite{Tond,Tond1} and $R^Q$  be the transversal curvature tensor  of $\nabla^Q\equiv\nabla$, which is  defined by $R^Q(X,Y)=[\nabla_X,\nabla_Y]-\nabla_{[X,Y]}$ for any $X,Y\in\Gamma TM$. Let $K^Q$ and ${\rm Ric}^Q $ be the transversal
sectional curvature and transversal Ricci operator with respect to $\nabla$, respectively.

Let $\Omega_B^r(\mathcal F)$ be the space of all {\it basic
$r$-forms}, i.e.,  $\omega\in\Omega_B^r(\mathcal F)$ if and only if
$i(X)\omega=0$ and $L_X\omega=0$ for any $X\in\Gamma T\mathcal F$, where $i(X)$ is the interior product. Then $\Omega^*(M)=\Omega_B^*(\mathcal F)\oplus \Omega_B^*(\mathcal F)^\perp$ \cite{Lop}.   It is well known that $\kappa_B$ is closed, i.e., $d\kappa_B=0$ \cite{Lop, PJ}, where  $\kappa_B$ is the basic part of the mean curvature form $\kappa$.  
Let $\bar *:\Omega_B^r(\mathcal F)\to \Omega_B^{q-r}(\mathcal F)$ be the basic Hodge star operator  given by
\begin{align*}
\bar *\omega = (-1)^{(n-q)(q-r)} *(\omega\wedge\chi_{\mathcal F}),\quad \omega\in\Omega_B^r(\mathcal F),
\end{align*}
where $\chi_{\mathcal F}$ is the characteristic form of $\mathcal F$ and $*$ is the Hodge star operator associated to $g$.  Let $\langle\cdot,\cdot\rangle$ be the pointwise inner product on $\Omega_B^r(\mathcal F)$, which is given by
\begin{align*}
\langle\omega_1,\omega_2\rangle \nu = \omega_1\wedge\bar * \omega_2,
\end{align*}
where $\nu$ is the transversal volume form such that $*\nu =\chi_{\mathcal F}$. 
 Let $\delta_B :\Omega_B^r (\mathcal F)\to \Omega_B^{r-1}(\mathcal F)$ be the operator defined by
\begin{align*}
\delta_B\omega = (-1)^{q(r+1)+1} \bar * (d_B-\kappa_B \wedge) \bar *\omega,
\end{align*}
where $d_B = d|_{\Omega_B^*(\mathcal F)}$.  Locally,   $\delta_{B}$ is  expressed by \cite{JR} 
\begin{equation}\label{2-2}
\delta_{B} = -\sum_a i(E_a) \nabla_{E_a} + i (\kappa_{B}^\sharp),
\end{equation}
 where $(\cdot)^\sharp$ is the dual vector field of $(\cdot)$ and $\{E_a\}_{a=1,\cdots,q}$ is a local orthonormal basic frame on $Q$.  
It is well known  \cite{PR} that $\delta_B$ is the formal adjoint of $d_B$ with respect to the global inner product $\ll\cdot,\cdot\gg$, which is defined by
\begin{align}\label{2-1}
\ll \omega_1,\omega_2\gg =\int_M \langle\omega_1,\omega_2\rangle\mu_M
\end{align}
for any compactly supported basic forms $\omega_1$ or $\omega_2$,
where $\mu_M =\nu\wedge\chi_{\mathcal F}$ is the volume form.   There exists a bundle-like metric  such that  the mean curvature form satisfies $\delta_B\kappa_B=0$ on compact manifolds \cite{DO,MMR,MA}.
 The  basic
Laplacian $\Delta_B$ acting on $\Omega_B^*(\mathcal F)$ is given by
\begin{equation*}
\Delta_B=d_B\delta_B+\delta_B d_B.
\end{equation*}
Now we define the bundle map $A_Y:\Gamma Q\to \Gamma Q$ for any $Y\in TM$ by
\begin{align}\label{A-operator}
A_Y s =L_Ys-\nabla_Ys,
\end{align}
where $L_Y s = \pi [Y,Y_s]$ for $\pi(Y_s)=s$ and $\pi:TM\to Q$ is the natural projection. 
It is well-known \cite{Kamber2} that for any  foliated vector field $Y$ ( i.e., $[Y,Z]\in \Gamma T\mathcal F$ for
all $Z\in \Gamma T\mathcal F$)
\begin{align*}
A_Y s = -\nabla_{Y_s}\pi(Y).
\end{align*}
So $A_Y$ depends only on $\pi(Y)$ and is a linear operator.  Moreover, $A_Y$ extends in an obvious way to tensors of any type on $Q$  \cite{Kamber2}.
Then we
have the generalized Weitzenb\"ock formula on $\Omega_B^*(\mathcal F)$ \cite{JU2}: for any $\omega\in\Omega_B^r(\mathcal  F),$
\begin{align}\label{2-3}
  \Delta_B \omega = \nabla_{\rm tr}^*\nabla_{\rm tr}\omega +
  F(\omega)+A_{\kappa_B^\sharp}\omega,
\end{align}
where $F(\omega)=\sum_{a,b}\theta^a \wedge i(E_b)R^Q(E_b,
 E_a)\omega$ and 
 \begin{align}\label{2-4}
\nabla_{\rm tr}^*\nabla_{\rm tr}\omega =-\sum_a \nabla^2_{E_a,E_a}\omega
+\nabla_{\kappa_B^\sharp}\omega.
\end{align}
Here $\theta^a$ is the dual form of $E_a$.
 The operator $\nabla_{\rm tr}^*\nabla_{\rm tr}$
is positive definite and formally self adjoint on the space of
basic forms \cite{JU2}. 
  If $\omega$ is a basic 1-form, then $F(\omega)^\sharp
 ={\rm Ric}^Q(\omega^\sharp)$.

Now, let $(M,g,\mathcal F, e^{-f}\nu)$ be a {\it weighted Riemannian foliation}, that is, Riemannian foliation endowed with a weighted transversal volume form $e^{-f}\nu$, where $f$ is  a basic function.  The formal adjoint operator $\delta_{B,f}$ of $d$ with respect to  volume form $e^{-f}\mu_M=\chi_\mathcal F \wedge e^{-f}\nu$ is given by
\begin{align}\label{2-6}
\delta_{B,f}\omega =e^f \delta_B (e^{-f}\omega) =\delta_B\omega + i(\nabla_{\rm tr} f)\omega
\end{align}
for any basic form $\omega$, where $\nabla_{\rm tr} f = \sum_a E_a(f)E_a$.  That is, for any basic forms $\omega\in\Omega_B^r(\mathcal F)$ and $\eta\in\Omega_B^{r+1}(\mathcal F)$,
\begin{align}
\int_M \langle d_B\omega,\eta\rangle e^{-f}\mu_M =\int_M \langle \omega,\delta_{B,f}\eta\rangle e^{-f}\mu_M.
\end{align}
The {\it weighted basic Laplacian}   $\Delta_{B,f}$ is defined by
\begin{align}\label{2-7}
\Delta_{B,f} = d_B\delta_{B,f} + \delta_{B,f}d_B.
\end{align}
From (\ref{2-6}),  we have
\begin{align}\label{2-8}
\Delta_{B,f} = \Delta_B + L_{\nabla_{\rm tr} f}.
\end{align}
Specially,   $\Delta_{B,f} =\Delta_B + i(\nabla_{\rm tr} f)d_B$ on $\Omega_B^0(\mathcal F)$.  Then we have the following.

\begin{lem}  Let $(M,g,\mathcal F)$ be a closed, connected Riemannian manifold with a foliation $\mathcal F$. If $(\Delta_{B,f} -\kappa_B)h\geq 0$ (or $\leq 0$) for any basic function $h$, then $h$ is constant.
\end{lem}
\begin{proof} The proof is similar to \cite[Lemma 2.1]{JLR}.  That is, let $f$ be a basic function on $M$.  
Since  $\Delta_B-\kappa_B =\Delta -\kappa$ on $\Omega_B^0(\mathcal F)$ \cite[Proposition 4.1]{PR},  we have that  on $\Omega_B^0(\mathcal F)$
\begin{align*}
\Delta_{B,f} -\kappa_B =\Delta_B -\kappa_B + i(\nabla_{\rm tr} f) d_B = \Delta -\kappa + i(\nabla f)d,
\end{align*}
 where $\Delta$ is the Laplace operator on $M$.   The operator of right hand side in the above is a second order elliptic  operator, by the maximum principle, the proof follows.
\end{proof}


\section{Harmonic maps}
\subsection {General facts}
Let $(M,  g,\mathcal F)$  and $(M', g',\mathcal F')$ be two foliated Riemannian manifolds and let $\phi:(M,g,\mathcal F)\to (M', g',\mathcal F')$ be a smooth foliated map,
i.e., $d\phi(T\mathcal F)\subset T\mathcal F'$. We define $d_T\phi:Q \to Q'$ by
\begin{align}
d_T\phi := \pi' \circ d \phi \circ \sigma,
\end{align}
where $\pi':TM'\to Q'$ is the natural projection.
Then $d_T\phi$ is a section in $ Q^*\otimes
\phi^{-1}Q'$, where $\phi^{-1}Q'$ is the pull-back bundle on $M$. Let $\nabla^\phi$
and $\tilde \nabla$ be the connections on $\phi^{-1}Q'$ and
$Q^*\otimes \phi^{-1}Q'$, respectively. 
The map  $\phi:(M, g,\mathcal F)\to (M', g',\mathcal F')$ is called {\it transversally totally geodesic} if it satisfies
\begin{align}
\tilde\nabla_{\rm tr}d_T\phi=0,
\end{align}
where $(\tilde\nabla_{\rm tr}d_T\phi)(X,Y):=(\tilde\nabla_{\sigma(X)} d_T\phi)(Y)$ for any $X,Y\in \Gamma Q$. Note that if $\phi:(M,g,\mathcal F)\to (M',g',\mathcal F')$ is transversally totally geodesic with $d\phi(Q)\subset Q'$, then, for any transversal geodesic $\gamma$ on $M$, $\phi\circ\gamma$ is also transversal geodesic.
From now on, we use $\nabla$ instead of all induced connections if we have no confusion.

The {\it transversal tension field} $\tau_{b}(\phi)$ of $\phi$ is defined by
\begin{align}\label{eq3-3}
\tau_{b}(\phi):={\rm tr}_{Q}(\nabla_{\rm tr} d_T\phi)=\sum_a (\nabla_{E_a}d_T\phi)(E_a).
\end{align}
 Let $\Omega$ be a compact domain of $M$. The {\it transversal energy}  of $\phi$ on $\Omega\subset
M$ is defined by
\begin{align}\label{eq2-4}
E_{B}(\phi;\Omega)={1\over 2}\int_{\Omega} | d_T \phi|^2\mu_{M}.
\end{align}
 The map $\phi$ is said to be {\it $(\mathcal F,\mathcal F')$-harmonic} \cite{DT} if $\phi$ is a critical point of the transversal energy functional $E_{B}(\phi)$.

Let $V\in\phi^{-1}Q'$ and $\phi_t$ be a foliated variation with $\phi_0=\phi$ and ${d\phi_t\over dt}|_{t=0}=V$.  Then we have the first variational formula \cite{JJ1} 
\begin{align}\label{3-5}
{d\over dt}E_{B}(\phi_t;\Omega)|_{t=0}=-\int_{\Omega} \langle V,\tau_{b}(\phi)-d_T\phi(\kappa_B^\sharp)\rangle \mu_{M},
\end{align}
where $\langle\cdot,\cdot\rangle$ is the pull-back metric on $\phi^{-1}Q'$.
Trivially,  $\phi$ is $(\mathcal F,\mathcal F')$-harmonic if and only if 
\begin{align*}
{\tau}_{b}(\phi)-d_T\phi(\kappa_B^\sharp)=0. 
\end{align*}
Let $\Omega_B^r(E)$ be the space of $E$-valued basic $r$-forms on $M$, where $E=\phi^{-1}Q'$.   We define $d_\nabla : \Omega_B^r(E)\to \Omega_B^{r+1}(E)$ by
\begin{align}
d_\nabla(\omega\otimes s)=d_B\omega\otimes s+(-1)^r\omega\wedge\nabla s
\end{align}
for any $s\in \Gamma E$ and $\omega\in\Omega_B^r(\mathcal F)$.
Let $\delta_\nabla$ be a formal adjoint of $d_\nabla$ with respect to the inner product induced from (\ref{2-1}).  
Trivially, we have the following remark.

\begin{rem}  Let $\phi:(M,\mathcal F)\to (M',\mathcal F')$ be a smooth foliated  map. Then
\begin{align*}
d_\nabla (d_T\phi)=0,\quad\delta_\nabla d_T\phi=-\tau_b (\phi) +d_T\phi(\kappa_B^\sharp).
\end{align*}
\end{rem}
Now we define the Laplacian $\Delta$ on $\Omega_B^*(E)$  by
\begin{align}\label{ee8}
\Delta =d_\nabla \delta_\nabla +\delta_\nabla d_\nabla.
\end{align}
Then the generalized Weitzenb\"ock type formula (\ref{2-3}) is extended to $\Omega_B^*(E)$ as follows \cite{JJ1}: for any $\Psi\in\Omega_B^r(E)$,
\begin{align}\label{eq4-6}
\Delta \Psi = \nabla_{\rm tr}^*\nabla_{\rm tr} \Psi
 + A_{\kappa_{B}^\sharp} \Psi + F(\Psi), 
\end{align}
where $ \nabla_{\rm tr}^*\nabla_{\rm tr}$, $A_X$ and $F(\Psi)$ are  naturally extended to $\Omega_B^r(E)$.
Moreover,  for any $ \Psi\in\Omega_B^r(E)$,
\begin{align}\label{weitzenbock}
\frac12\Delta_B|\Psi |^{2}
=\langle\Delta \Psi, \Psi\rangle -|\nabla_{\rm tr} \Psi|^2-\langle A_{\kappa_{B}^\sharp}\Psi, \Psi\rangle -\langle F(\Psi),\Psi\rangle.
\end{align}
Then we have the generalized Weitzenb\"ock type formula as follows.
\begin{prop}\label{th2} \cite{JJ1}
Let $\phi:(M, g,\mathcal F) \to (M', g', \mathcal F')$ be a smooth foliated map. Then 
\begin{align}\label{3-11}
\frac12\Delta_B| d_T \phi |^{2}
=  - |\nabla_{\rm tr} d_T \phi|^2  -\langle F(d_T\phi),d_T\phi\rangle-\langle d_\nabla \tau_b(\phi),d_T\phi\rangle +\frac12{\kappa_B^\sharp}(|d_T\phi|^2),
\end{align}
where
\begin{align}\label{3-12}
\langle F(d_T\phi),d_T\phi\rangle&=\sum_a g_{Q'}(d_T \phi({\rm Ric^{Q}}(E_a)),d_T \phi(E_a)) \notag\\
&-\sum_{a,b}g_{Q'}( R^{Q'}(d_T \phi(E_b), d_T \phi(E_a))d_T \phi(E_a), d_T \phi(E_b)).
\end{align}
\end{prop}
\begin{proof}  By $\Psi=d_T\phi$ in (\ref{weitzenbock}), it is trivial from Remark 3.1.
\end{proof}
Now,  we define  the {\it transversal Bakry-\'Emery-Ricci curvature } ${\rm Ric}_f^Q$ by 
\begin{align}\label{2-9}
{\rm Ric}_f^Q = {\rm Ric}^Q + {\rm Hess}_T f,
\end{align}
where ${\rm Hess}_T f :=\nabla_{\rm tr}d_B f$ is the transversal Hessian, that is, ${\rm Hess}_Tf(X,Y):=(\nabla _{\sigma(X)} d_B f)(\sigma(Y))$ for any  vectors $X, Y\in \Gamma Q$.  
\begin{thm}  Let $\phi:(M,g,\mathcal F,e^{-f}\nu) \to (M',g',\mathcal F')$ be a smooth foliated map. Then 
\begin{align*}
\frac12\Delta_{B,f}|d_T\phi|^2 =-|\nabla_{\rm tr}d_T\phi|^2  - \langle F_f(d_T\phi),d_T\phi\rangle  -\langle d_\nabla(\bar\tau_{b,f}(\phi)),d_T\phi\rangle+ \frac12 \kappa_B^\sharp (|d_T\phi|^2),
\end{align*}
where   $\bar\tau_{b,f}(\phi) =\tau_b(\phi) -d_T\phi(\nabla_{\rm tr}f)$ and
\begin{align}\label{3-13}
\langle F_f(d_T\phi),d_T\phi\rangle&=\sum_a g_{Q'}(d_T \phi({\rm Ric}_f^{Q}(E_a)),d_T \phi(E_a)) \notag\\
&-\sum_{a,b}g_{Q'}( R^{Q'}(d_T \phi(E_b), d_T \phi(E_a))d_T \phi(E_a), d_T \phi(E_b)).
\end{align}
\end{thm}
\begin{proof}
From (\ref{2-9}), we know 
\begin{align*}
\langle F(d_T\phi),d_T\phi\rangle = \langle F_f (d_T\phi),d_T\phi\rangle - \langle d_T\phi({\rm Hess}_Tf),d_T \phi\rangle,
\end{align*}
where  $\langle d_T\phi({\rm Hess}_Tf),d_T \phi\rangle := \sum_a g_{Q'} (d_T\phi(\nabla_{E_a} \nabla_{\rm tr} f),d_T\phi(E_a))$. Hence from (\ref{2-8}) and  (\ref{3-11}),
\begin{align}\label{3-14}
\frac12\Delta_{B,f}|d_T\phi|^2=&-|\nabla_{\rm tr}d_T\phi|^2 - \langle F_f(d_T\phi),d_T\phi\rangle+\langle\nabla_{\kappa_B^\sharp} d_T\phi,d_T\phi\rangle -\langle d_\nabla(\tau_b(\phi)),d_T\phi\rangle  \notag\\
&+\langle d_T\phi({\rm Hess}_Tf),d_T\phi\rangle+\frac12 \langle \nabla_{\rm tr}|d_T\phi|^2,\nabla_{\rm tr} f\rangle.
\end{align}
Note that $(\nabla_{\rm tr} d_T\phi)(X,Y) =(\nabla_{\rm tr} d_T\phi)(Y, X)$ for any  vector fields $X,Y \in \Gamma Q$ . Hence if we choose a local orthonormal basic frame $\{E_a\}$ such that $\nabla E_a=0$ at $x\in M$, then
\begin{align*}
\frac12 \langle \nabla_{\rm tr}|d_T\phi|^2,\nabla_{\rm tr} f\rangle &=\sum_a \langle (\nabla_{\nabla_{tr} f} d_T\phi)(E_a),d_T\phi(E_a)\rangle\\
&=\sum_a\langle (\nabla_{E_a} d_T\phi)(\nabla_{\rm tr}f),d_T\phi(E_a)\rangle\\
&=\sum_a\langle \nabla_{E_a} d_T\phi(\nabla_{\rm tr}f),d_T\phi(E_a)\rangle -\sum_a\langle d_T\phi(\nabla_{E_a}\nabla_{\rm tr}f),d_T\phi(E_a)\rangle\\
&=\sum_a\langle \theta^a\wedge\nabla_{E_a} d_T\phi(\nabla_{\rm tr}f),d_T\phi\rangle -\sum_a\langle d_T\phi(\nabla_{E_a}\nabla_{\rm tr}f),d_T\phi(E_a)\rangle\\
&=\langle d_\nabla(d_T\phi(\nabla_{\rm tr}f)),d_T\phi\rangle -\langle d_T\phi({\rm Hess}_Tf),d_T\phi\rangle.
\end{align*}
That is,
\begin{align}\label{3-15}
\langle d_T\phi({\rm Hess}_Tf),d_T\phi\rangle+\frac12 \langle \nabla_{\rm tr}|d_T\phi|^2,\nabla_{\rm tr} f\rangle=\langle d_\nabla(d_T\phi(\nabla_{\rm tr}f)),d_T\phi\rangle.
\end{align}
From  (\ref{3-14}) and (\ref{3-15}), the proof follows.
\end{proof}
\subsection{$f$-harmonic maps}

A smooth foliated map $\phi:(M,g,\mathcal F,e^{-f}\nu)\to (M',g',\mathcal F')$ is said to be {\it transversally $f$-harmonic} if  $\phi$ is a solution of  the Eular-Largrange  equation $\tau_{b,f}(\phi)=0$, where  $\tau_{b,f}:={\rm tr}_Q (\nabla_{\rm tr}( e^{-f} d_T\phi))$ is the transversal $f$-tension field, that is,
\begin{align}
\tau_{b,f}(\phi)  =( \tau_b(\phi) - d_T\phi(\nabla_{\rm tr}f))e^{-f}. 
\end{align}
The map $\phi$ is said to be {\it $(\mathcal F,\mathcal F')_f$-harmonic map} if $\phi$ is a critical point of  the {\it transversal $f$-energy functional} $E_{B,f}(\phi)$ given by
\begin{align}\label{3-16}
E_{B,f}(\phi,\Omega) = \frac12\int_\Omega |d_T\phi|^2 e^{-f}\mu_M.
\end{align}
Remark that if $f$ is constant, then a transversally $f$-harmonic and $(\mathcal F,\mathcal F')_f$-harmonic map are transversally harmonic and $(\mathcal F,\mathcal F')$-harmonic map, respectively.
\begin{thm} $(${\rm The first variational formula}$)$ \label{th4}
Let $\phi:(M, g, \mathcal F)\to (M', g', \mathcal F')$
be a smooth foliated map and $\{\phi_t\}$ be a smooth foliated variation of $\phi$ supported in a compact domain $\Omega$. Then
\begin{align*}
{d\over dt}E_{B,f}(\phi_t;\Omega)|_{t=0}=-\int_{\Omega} \langle V, e^f{\tau}_{b,f}(\phi)-d_T\phi(\kappa_B^\sharp)\rangle e^{-f} \mu_{M},
\end{align*}
where $V$ is the variation vector field of $\phi_t$.
\end{thm}
\begin{proof}  It is trivial from  \cite[Theorem 3.7]{JU3}.
\end{proof}
From Theorem 3.4,  the map $\phi:M\to M'$ is  $(\mathcal F,\mathcal F')_f$-harmonic map if and only if  
\begin{align}\label{3-17}
\tilde\tau_{b,f}(\phi) := e^f\tau_{b,f}(\phi) -d_T\phi(\kappa_B^\sharp) =0.
\end{align}
\begin{rem} A $(\mathcal F,\mathcal F')_f$-harmonic map on $(M,g,\mathcal F)$ is a $(\mathcal F,\mathcal F')$-harmonic map on a weighted Riemannian foliation $(M,g,\mathcal F,e^{-f}\nu)$.
\end{rem}
In general,  $(\mathcal F,\mathcal F')_f$-harmonic map and transversally $f$-harmonic map are not equivalent  unless $\mathcal F$ is  minimal. For more information about the transversally $f$-harmonic map, see \cite{CW}.

\begin{lem}

(1)  If $\phi:M\to M'$ is  a transversally $f$-harmonic map (i.e.$\tau_{b,f}(\phi)=0$),  then
\begin{align}\label{3-18}
|d_T\phi|(\Delta_{B,f}-\kappa_B^\sharp)|d_T\phi|
=|d_{B}|d_T\phi||^{2}-|\nabla_{\rm tr}d_T\phi|^{2}-\langle F_f(d_T\phi),d_T\phi\rangle.
\end{align}

(2) If  $\phi:M\to M'$ be a  $(\mathcal F,\mathcal F')_f$-harmonic map (i.e., $\tilde\tau_{b,f}(\phi)=0$), then
\begin{align}\label{3-19}
|d_T\phi| (\Delta_{B,f}-\kappa_B^\sharp)|d_T\phi|
=&|d_{B}|d_T\phi||^{2}-|\nabla_{\rm tr}d_T\phi|^{2}-\langle F_f(d_T\phi),d_T\phi\rangle\notag\\
&-\langle d_\nabla i(\kappa_{B}^\sharp)d_T\phi,d_T \phi\rangle.
\end{align}
\end{lem}
\begin{proof}
By a simple calculation, we have
\begin{align*}
\frac12\Delta_{B,f}| d_T \phi |^{2}
=|d_T\phi|\Delta_{B,f}|d_T\phi|-|d_{B}|d_T\phi||^{2}.
\end{align*}
Hence  the proofs follow from Theorem 3.3 and (\ref{3-17}). 
\end{proof}

Then we have the following.

\begin{thm} Let $(M,g,\mathcal F,e^{-f}\nu)$ be a weighted Riemannian foliation on a closed manifold $M$ with ${\rm Ric}_{f}^Q\geq 0$ and $(M',g',\mathcal F')$ be a Riemannian foliation with $K^{Q'}\leq 0$. Then a transversally $f$-harmonic (or $(\mathcal F,\mathcal F')_f$-harmonic)  map $\phi:M\to M'$ is always transversally totally geodesic. 
In addition, 

(1) if ${\rm Ric}_f^Q>0$ at some point, then $\phi$ is transversally constant.

(2) if  $K^{Q'}<0$, then $\phi$ is transversally constant or $\phi(M)$  is a transversally geodesic closed curve.
\end{thm}
\begin{proof}  
By the first Kato's inequality \cite{BE1}, we have
\begin{align}\label{3-20}
|\nabla_{\rm tr}d_T\phi|\geq|d_{B}|d_T\phi||.
\end{align}
(a)  Let $\phi:M\to M'$ be a  transversally $f$-harmonic map. That is,    $\tau_{b,f}(\phi)=0$.   From  (\ref{3-18}) and (\ref{3-20}), we have
\begin{align}\label{3-21}
|d_T\phi| (\Delta_{B,f}-\kappa_B^\sharp)|d_T\phi|\leq -\langle F_f(d_T\phi),d_T\phi\rangle.
\end{align}
From the assumptions ${\rm Ric}_f^Q\geq 0$ and $K^{Q'}\leq 0$,   $\langle F_f(d_T\phi),d_T\phi\rangle \geq 0$ and so
\begin{align*}
(\Delta_{B,f} -\kappa_B^\sharp)|d_T\phi|\leq 0.
\end{align*}
By Lemma 2.1,  $|d_T\phi|$ is constant.  Again, from  (\ref{3-18}), we have
\begin{align}\label{3-22}
|\nabla_{\rm tr} d_T \phi|^2+\langle F_f(d_T\phi),d_T\phi\rangle=0.
\end{align}
Since  $\langle F_f(d_T\phi),d_T\phi\rangle \geq 0$,
from (\ref{3-22}), we have
\begin{align}\label{3-23}
|\nabla_{\rm tr} d_T \phi|^2=0 \quad\textrm{and}\quad \langle F_f(d_T\phi),d_T\phi\rangle=0.
\end{align}
Thus, $\nabla_{\rm tr}d_T\phi=0$, i.e., $\phi$ is transversally totally geodesic.

Furthermore, from (\ref{3-13}) and (\ref{3-23}), we get
\begin{align}\label{3-24}
\left\{
  \begin{array}{ll}
    g_{Q'}(d_T\phi({\rm Ric}_f^{Q}(E_a)),d_T\phi(E_a))= 0,\\\\
    g_{Q'}(R^{Q'}(d_T\phi(E_a),d_T\phi(E_b))d_T\phi(E_a),d_T\phi(E_b))= 0
  \end{array}
\right.
\end{align}
for any indices $a$ and $b$.
If ${\rm Ric}_f^{Q}$ is positive at some point, then $d_T\phi=0$, i.e., $\phi$ is transversally constant, which proves (1). For the statement (2), if the rank of $d_T\phi \geq2$, then there exists a point $x\in M$ such that at least two linearly independent vectors at $\phi(x)$, say, $d_T\phi(E_1)$ and $d_T\phi(E_2)$.
Since  $K^{Q'}<0$,
\begin{align*}
g_{Q'}(R^{Q'}(d_T\phi(E_1),d_T\phi(E_2))d_T\phi(E_2),d_T\phi(E_1))<0,
\end{align*}
which contradicts (\ref{3-24}). Hence the rank of $d_T\phi <2$, that is, the rank of $d_T\phi$ is zero or one everywhere. If the rank of $d_T\phi$ is zero, then $\phi$ is transversally constant. If the rank of $d_T\phi$ is one, then $\phi(M)$ is  a transversally  geodesic closed curve.

(b) Let $\phi:M\to M'$ be a $(\mathcal F,\mathcal F')_f$-harmonic map.  Then 
\begin{align}\label{3-25-1}
\delta_\nabla (e^{-f}d_T\phi) =-\tilde\tau_{b,f}(\phi)=0.
\end{align}
Hence from (\ref{3-19}) and (\ref{3-20}), we get
\begin{align}\label{3-25}
|d_T\phi| \Delta_{B,f}|d_T\phi|  \leq -\langle F_f(d_T\phi),d_T\phi\rangle -\langle d_\nabla i(\kappa_B^\sharp)d_T\phi,d_T\phi\rangle+|d_T\phi| \kappa_B^\sharp(|d_T\phi|).
\end{align}
By the curvature assumptions,   $\langle F_f(d_T\phi),d_T\phi\rangle \geq 0$. So we get from (\ref{3-25}),
\begin{align}\label{3-26}
|d_T\phi|\Delta_{B,f}|d_T\phi|
\leq-\langle d_\nabla i(\kappa_{B}^\sharp)d_T\phi,d_T \phi\rangle+|d_T\phi|\kappa_{B}^\sharp(|d_T\phi|).
\end{align}
Integrating (\ref{3-26}) with the weighted measure, we have
\begin{align}\label{3-27}
\int_{M}\langle&|d_T\phi|,\Delta_{B,f}|d_T\phi|\rangle e^{-f}\mu_{M}
\leq-\int_{M}\langle d_\nabla i(\kappa_{B}^\sharp)d_T\phi,d_T \phi\rangle e^{-f}\mu_{M}+\int_{M}|d_T\phi|\kappa_{B}^\sharp(|d_T\phi|)e^{-f}\mu_{M}.
\end{align}
From (\ref{3-25-1}), we get
\begin{align}\label{3-28}
\int_{M}\langle d_\nabla i(\kappa_{B}^\sharp)d_T\phi,d_T \phi\rangle e^{-f}\mu_{M}
=\int_{M}\langle i(\kappa_{B}^\sharp)d_T\phi,\delta_\nabla (e^{-f}d_T \phi)\rangle \mu_{M}
=0.
\end{align}
Since $M$ is closed,   we have
\begin{align}\label{3-30-1}
\int_{M}|d_T\phi|\kappa_{B}^\sharp(|d_T\phi|)e^{-f}\mu_{M}
&\leq B\int_M \langle d_B |d_T\phi|^2, \kappa_B\rangle \mu_M,
\end{align}
where $B$ is a positive constant.
If we choose a bundle-like metric $g$ such that $\delta_{B}\kappa_{B}=0$, then  from (\ref{3-30-1}),
\begin{align}\label{3-29}
\int_{M}|d_T\phi|\kappa_{B}^\sharp(|d_T\phi|)e^{-f}\mu_{M}  \leq 0.
\end{align}
Note that  $ \int_{M}\langle|d_T\phi|,\Delta_{B,f}|d_T\phi|\rangle e^{-f}\mu_{M}=\int_M |d_B|d_T\phi||^2 e^{-f}\mu_M\geq 0$. So from (\ref{3-27})$\sim$(\ref{3-29}), we get
\begin{align*}
\int_{M}\langle |d_T\phi|,\Delta_{B,f}|d_T\phi|\rangle e^{-f}\mu_{M}=0,  
\end{align*}
which yields  $ d_{B}|d_T\phi|=0$.  That is, $|d_T\phi|$ is constant.
From (\ref{3-19}), we have
\begin{align}\label{3-30}
0=&-|\nabla_{\rm tr}d_T\phi|^{2}-\langle d_\nabla i(\kappa_{B}^\sharp)d_T\phi,d_T \phi\rangle -\langle F_f(d_T\phi),d_T\phi\rangle.
\end{align}
From  (\ref{3-28}) and (\ref{3-30}), by integrating  we get
\begin{align}\label{3-31}
\int_{M}|\nabla_{\rm tr}d_T\phi|^{2} e^{-f}\mu_{M}
+\int_{M}\langle F_f(d_T\phi),d_T\phi\rangle e^{-f}\mu_{M}=0.
\end{align}
Since  $\langle F_f(d_T\phi),d_T\phi\rangle \geq 0$,
from (\ref{3-31}), we have
\begin{align}\label{3-32}
|\nabla_{\rm tr} d_T \phi|^2=0 \quad\textrm{and}\quad \langle F_f(d_T\phi),d_T\phi\rangle=0.
\end{align}
Hence  the remaining part of the proof is same  to that of (a). So we omit the remaining part of the proof.
\end{proof}
\begin{rem} If $f$ is constant, then transversally $f$-harmonic map (resp. $(\mathcal F,\mathcal F')_f$-harmonic) is just transversally harmonic (resp. $(\mathcal F,\mathcal F')$-harmonic). So Theorem 3.7 holds for  transversally harmonic (cf. \cite{JJ1}) and $(\mathcal F,\mathcal F')$-harmonic map.
\end{rem}

\section{Liouville type theorems}

In this section, we investigate the Liouville type theorems for  transversally $f$-harmonic map and $(\mathcal F,\mathcal F')_{f}$-harmonic map  on weighted Riemannian foliations. 

Let  $\rho(y)$ be the distance between leaves through a fixed point $x_{0}$ and $y$  and $B_{l}=\{y\in M|\rho(y)\leq l\}$.
Let $\omega_{l}$ be the Lipschitz continuous basic function such that
\begin{align*}
\left\{
  \begin{array}{ll}
    0\leq\omega_{l}(y)\leq1 \quad {\rm for \, any} \, y\in M\\
    {\rm supp}\, \omega_{l}\subset B_{2l}\\
    \omega_{l}(y)=1 \quad {\rm for \, any} \,  y\in B_{l}\\
    \lim\limits_{l\rightarrow\infty}\omega_{l}=1\\
    |d\omega_{l}|\leq\frac{C}{l} \quad\textrm {almost  everywhere  on $M$},
  \end{array}
\right.
\end{align*}
where $C$ is positive constant \cite{Y1}. Therefore, $\omega_{l}\eta$ has compact support for any basic form $\eta\in\Omega_{B}^{*}(\mathcal F)$.  Let $f_\kappa$ be a solution of  the following inequality $\langle\kappa_B,df\rangle\leq 0$. Trivially, there exists such functions on taut folations. 
\begin{lem} Let $(M,g,\mathcal F,e^{-f_\kappa}\nu)$ be a weighted Riemannian foliation on a complete Riemannian manifold whose  all leaves are compact  and the mean curvature form is bounded.   Let  $\phi: (M,g,\mathcal F)\to (M',g',\mathcal F')$ be a smooth foliated map of $E_{B,f_\kappa}(\phi) <\infty$. Then 
\begin{align}\label{4-1}
\lim\limits_{l\rightarrow\infty}\int_{M}\langle \omega_{l}^{2}|d_T\phi|,\kappa_{B}^\sharp(|d_T\phi|)\rangle e^{-f_\kappa}\mu_{M}\leq 0.
\end{align}
Moreover, if $\phi$ is $(\mathcal F,\mathcal F')_{f_\kappa}$-harmonic, then
\begin{align}\label{4-2}
\lim\limits_{l\rightarrow\infty}\int_{M}\langle d_\nabla i(\kappa_{B}^\sharp)d_T\phi,\omega_{l}^{2}d_T \phi\rangle e^{-f_\kappa}\mu_{M}=0.
\end{align}
\end{lem}
\begin{proof} 
Let $g$ be a bundle-like metric such that $\delta_{B}\kappa_{B}=0$. Then  we get
\begin{align*}
\int_{M}\langle \omega_{l}^{2}|d_T\phi|,\kappa_{B}^{\sharp}(|d_T\phi|)\rangle e^{-f_\kappa}\mu_{M}
=&\frac12\int_M \langle d_B |d_T\phi|^2,\omega_l^2\kappa_B\rangle e^{-f_\kappa}\mu_M\notag\\
=&\frac{1}{2}\int_{M}\delta_{B,f}(\omega_l^2 \kappa_{B})|d_T\phi|^{2}e^{-f_\kappa}\mu_{M}\notag \\
=&-\int_M \Big(\langle d_B\omega_l ,\omega_l\kappa_B\rangle -\frac12 i(\nabla_{\rm tr}f_\kappa)\omega_l^2\kappa_B\Big) |d_T\phi|^2 e^{-f_\kappa}\mu_M.
\end{align*}
Since $i(\nabla_{\rm tr}f_\kappa)\kappa_B=i(\kappa_B^\sharp)d_B f_\kappa\leq 0$, then from the above equation,
\begin{align*}
\int_{M}\langle \omega_{l}^{2}|d_T\phi|,\kappa_{B}^{\sharp}(|d_T\phi|)\rangle e^{-f_\kappa}\mu_{M}
\leq-\int_{M}\langle d_B\omega_l,\omega_l\kappa_B\rangle |d_T\phi|^2 e^{-f_\kappa}\mu_M.
\end{align*}
By using the Cauchy-Schwartz inequality and letting $l\to \infty$, we get
\begin{align*}
\Big|\int_M \langle d_B\omega_l,\omega_l\kappa_B\rangle |d_T\phi|^2 e^{-f_\kappa}\mu_M\Big| &\leq {C\over l}{\rm max} |\kappa_B|\int_M\omega_l |d_T\phi|^2 e^{-f_\kappa}\mu_M \longrightarrow 0,
\end{align*}
which proves (\ref{4-1}).
For the proof of (\ref{4-2}), let $\phi$ be  $(\mathcal F,\mathcal F')_{f_\kappa}$-harmonic. Then 
\begin{align*}
\delta_\nabla (\omega_{l}^{2}e^{-f_\kappa}d_T \phi)&=\omega_l^2 \delta_\nabla(e^{-f_\kappa}d_T\phi) -i(d_{B}\omega_{l}^{2})e^{-f_\kappa}d_T \phi\\
&=-i(d_{B}\omega_{l}^{2})e^{-f_\kappa}d_T \phi\\
&=-2\omega_{l}e^{-f_\kappa}i(d_{B}\omega_{l}) d_T \phi.
\end{align*}
By using the inequality
\begin{align*}
|X^{\flat}\wedge d_T\phi|^{2}+|i(X)d_T\phi|^{2}=|X|^{2}|d_T\phi|^{2}
\end{align*}
for any vector $X$,  we get
\begin{align*}
\bigg{|}\int_{M}\langle d_\nabla i(\kappa_{B}^\sharp)d_T\phi,\omega_{l}^{2}d_T \phi\rangle e^{-f_\kappa}\mu_{M}\bigg{|}
=&\bigg{|}\int_{M}\langle i(\kappa_{B}^\sharp)d_T\phi,\delta_\nabla (\omega_{l}^{2}e^{-f_\kappa}d_T \phi)\rangle \mu_{M}\bigg{|}\\
=&\bigg{|}\int_{M}\langle i(\kappa_{B}^\sharp)d_T\phi, -2\omega_{l}i(d_{B}\omega_{l})d_T \phi\rangle e^{-f_\kappa}\mu_{M}\bigg{|}\notag \\
\leq&2\int_{M}\omega_{l}|i(\kappa_{B}^\sharp)d_T\phi||i(d_{B}\omega_{l})d_T \phi| e^{-f_\kappa}\mu_{M}\notag \\
\leq&2\int_{M}\omega_{l}|\kappa_{B}||d_{B}\omega_{l}||d_T\phi|^{2}e^{-f_\kappa}\mu_{M}\notag \\
\leq&\frac{2C}{l}\max|\kappa_{B}|\int_{M}\omega_{l}|d_T\phi|^{2}e^{-f_\kappa}\mu_{M}.
\end{align*}
By letting $l\to\infty$,  $E_{B,f_\kappa}(\phi)<\infty$ implies the proof of (\ref{4-2}).
\end{proof}
Then we have the following Liouville type theorem for  $f$-harmonic maps. 
\begin{thm}
Let $(M,g,\mathcal F,e^{-f_\kappa}\nu)$ be a weighted Riemannian foliation on a complete Riemannian manifold whose  all leaves are compact and the mean curvature form is bounded. Let $(M',g',\mathcal F')$ be a foliated Riemannian manifold with $K^{Q'}\leq 0$.
Let $\phi : M \rightarrow M'$ be a transversally $f_\kappa$-harmonic (or $(\mathcal F,\mathcal F')_{f_\kappa}$-harmonc) map   of $E_{B,f_\kappa}(\phi)<\infty$.   Then  

(1) if ${\rm Ric}_{f_\kappa}^Q\geq 0$, then $\phi$ is transversally totally geodesic.

(2) if ${\rm Ric}_{f_\kappa}^Q \geq 0$  and  either ${\rm Ric}_{f_\kappa}^Q >0$ at some point or $\int_M e^{-f_\kappa}\mu_M= \infty$, then $\phi$ is transversally constant.

\end{thm}

\begin{proof}
(1) Since  ${\rm Ric}_f^Q \geq 0$ and $K^{Q'}\leq 0$,  from (\ref{3-13}),  $\langle F_f(d_T\phi),d_T\phi\rangle \geq 0$.   If $\phi$  is  transversally $f_\kappa$-harmonic, then from Lemma 3.5 (1) and the first Kato's inequality (\ref{3-20}),  we have
\begin{align}\label{4-4}
|d_T\phi|\Delta_{B,f} |d_T\phi| \leq  |d_T\phi|\kappa_B^\sharp(|d_T\phi|).
\end{align}
If $\phi$ is $(\mathcal F,\mathcal F')_{f_\kappa}$-harmonic, then from Lemma 3.5 (2) and the first Kato's inequality (\ref{3-20}),  we have
\begin{align}\label{4-4-1}
|d_T\phi|\Delta_{B,f} |d_T\phi| \leq  -\langle d_\nabla i(\kappa_{B}^\sharp)d_T\phi,d_T \phi\rangle+ |d_T\phi|\kappa_B^\sharp(|d_T\phi|).
\end{align}
Multiplying (\ref{4-4}) and (\ref{4-4-1}) by $\omega_{l}^{2}$ and integrating by parts,  from Lemma 4.1, whether $\phi$ is transversally $f_\kappa$-harmonic or $(\mathcal F,\mathcal F')_{f_\kappa}$-harmonic, we get
\begin{align}\label{4-5}
\lim_{l\to\infty}\int_{M}&\langle \omega_{l}^{2}|d_T\phi|,\Delta_{B,f_\kappa}|d_T\phi|\rangle e^{-f_\kappa}\mu_{M}\leq 0.
\end{align}
On the other hand, by the Cauchy-Schwarz inequality, we have
\begin{align}\label{4-6}
\int_{M}\langle \omega_{l}^{2}&|d_T\phi|,\Delta_{B,f}|d_T\phi|\rangle e^{-f}\mu_{M}\notag\\
 =&\int_M\omega_l^2 |d_B|d_T\phi||^2 e^{-f}\mu_M+ 2\int_M \omega_l\langle |d_T\phi|d_B\omega_l,d_B|d_T\phi|\rangle e^{-f}\mu_M\notag\\
\geq& \int_{M}\omega_{l}^{2}|d_{B}|d_T\phi||^{2}e^{-f}\mu_{M}-2\int_{M}\omega_{l}|d_T\phi||d_{B}\omega_{l}||d_{B}|d_T\phi||e^{-f}\mu_{M}.
\end{align}
By using  the inequality $A^2 +B^2 \geq 2AB$,  the second term  in (\ref{4-6}) implies 
\begin{align}\label{4-7}
2\int_{M}\omega_{l}|d_{B}\omega_{l}||d_T\phi||d_{B}|d_T\phi|| e^{-f}\mu_{M}\leq{C\over l} \int_M\Big(\omega_l^2  |d_T\phi|^2 e^{-f} +|d_B|d_T\phi||^2e^{-f}\Big)\mu_M.
\end{align}
From (\ref{4-5}) $\sim$(\ref{4-7}) and Fatou's inequality, it is trivial that $|d_{B}|d_T\phi||\in L^{2}(e^{-f_\kappa})$, that is, $\int_M |d_B|d_T\phi||^2 e^{-f_\kappa}\mu_M<\infty$. So letting $l\rightarrow\infty$, we get from (\ref{4-7})
\begin{align}\label{4-8}
\lim\limits_{l\rightarrow\infty}\int_{M}\omega_{l}|d_{B}\omega_{l}||d_T\phi||d_{B}|d_T\phi|| e^{-f_\kappa}\mu_{M}=0.
\end{align}
From (\ref{4-6}) and (\ref{4-8}),  by letting $l\to \infty$, we have
\begin{align}\label{4-9}
\lim\limits_{l\rightarrow\infty}\int_{M}\langle& \omega_{l}^{2}|d_T\phi|,\Delta_{B,f_\kappa}|d_T\phi|\rangle e^{-f_\kappa}\mu_{M}
\geq\int_{M}|d_{B}|d_T\phi||^{2} e^{-f_\kappa}\mu_{M}.
\end{align}
From (\ref{4-5}) and (\ref{4-9}), we get
\begin{align}\label{4-9-1}
\int_{M}|d_{B}|d_T\phi||^{2} e^{-f_\kappa}\mu_{M}\leq 0,
\end{align}
which implies  $d_B|d_T\phi|=0$, that is, $|d_T\phi|$ is constant.  

(i) If $\phi$ is transversally $f_\kappa$-harmonic,  then from (\ref{3-18}) we have
\begin{align}\label{4-10}
\nabla_{tr}d_T\phi =0,\quad \langle F_f(d_T\phi),d_T\phi\rangle=0.
\end{align}
So $\phi$ is transversally totally geodesic. 

(ii)  If $\phi$ is $(\mathcal F,\mathcal F')_{f_\kappa}$-harmonic, then from (\ref{3-19}), we get
\begin{align}\label{4-10-1}
|\nabla_{tr}d_T\phi|^2 + \langle d_\nabla i(\kappa_B^\sharp)d_T\phi,d_T\phi\rangle = -\langle F_{f_\kappa}(d_T\phi),d_T\phi\rangle \leq 0.
\end{align}
By multiplying $\omega_l^2$ and integrating,  from (\ref{4-2}) and (\ref{4-10-1}), we have
\begin{align*}
\int_M |\nabla_{tr}d_T\phi|^2 e^{-f_\kappa}\mu_M \leq 0,
\end{align*}
which implies  (\ref{4-10}).  Hence $\phi$ is transversally totally geodesic.

(2) From (\ref{3-13}) and (\ref{4-10}), we get 
\begin{align*}
0=\langle F_{f_\kappa}(d_T\phi),d_T\phi\rangle\geq \sum_a g_{Q'}(d_T \phi({\rm Ric}_{f_\kappa}^{Q}(E_a)),d_T \phi(E_a))\geq 0.
\end{align*}
That is,
\begin{align}\label{4-11}
\sum_a g_{Q'}(d_T \phi({\rm (Ric}_{f_\kappa}^{Q})(E_a)),d_T \phi(E_a)) e^{-f_\kappa}\mu_{M}=0.
\end{align}
  If ${\rm Ric}_{f_\kappa}^{Q}>0$ at some point,  from (\ref{4-11}),  $d_T \phi=0$, that is,  $\phi$ is transversally constant.  
If $\int_M e^{-f_\kappa}\mu_M=\infty$, then  $E_{B,f_\kappa}(\phi)=\frac12 |d_T\phi|^2 \int_M e^{-f_\kappa}\mu_M<\infty$  implies  $d_T\phi=0$, that is, $\phi$ is transversally constant. 
\end{proof}

If $f$ is constant, then  $\Delta_{B,f}=\Delta_B$, ${\rm Ric}_f^Q ={\rm Ric^Q}$, and  $(\mathcal F,\mathcal F')_f$-harmonic map is just $(\mathcal F,\mathcal F')$-harmonic map. Hence we have the following corollary.
\begin{cor}\label{co6}
Let $(M,g,\mathcal F)$ be a complete foliated Riemannian manifold whose all leaves are compact and the mean curvature form is bounded.
Let $(M',g',\mathcal F')$ be a foliated Riemannian manifold with  $K^{Q'}\leq 0$. Assume that the transversal Ricci curvature ${\rm Ric^{Q}}$ of $M$ satisfies ${\rm Ric^{Q}}\geq 0$ for all points and ${\rm Ric^{Q}}>0$ at some point.
Then any $(\mathcal F,\mathcal F')$-harmonic map $\phi : (M,g,\mathcal F) \rightarrow (M', g',\mathcal F')$ of $E_{B}(\phi)<\infty$ is transversally  constant.
\end{cor}

\begin{rem}  (1)  Theorem 4.2  can be found for the point foliation in \cite{RV}.

(2)  Corollary \ref{co6}  for the transversally harmonic map has been studied by Fu and Jung \cite{FJ}.

\end{rem}
\section{The stress energy tensor}
Let $\phi:(M,g,\mathcal F)\to (M',g',\mathcal F')$ be a smooth foliated map.   In this section, we calculate the rate of  change  of the transversal energy of $\phi$ when the metric $g_Q$ is changed.  Let $g_Q (t)$ be the variation of $g_Q$ with $g_Q(0)=g_Q$.  With a transversal coordinate $\{y^a\}(a=1,\cdots,q)$, the metric $g_Q$ is written by $g_Q(t)=\sum_{a,b} g_{ab}(t)dy^a dy^b$. We put $h= {d \over dt}g_Q(t)|_{t=0}$, a symmetric 2-tensor on $M$. That is,  in local coordinates, $h$ can be written by
\begin{align*}
h = \sum_{a,b}{d g_{ab}\over dt}|_{t=0} dy^a dy^b=\sum_{a,b}h_{ab} dy^a dy^b.
\end{align*}
 Then we have the following.
\begin{lem}  Let $\mu_M (t)$ be the volume form with respect to the metric $g(t) = g_L + g_Q(t)$. Then
\begin{align*}
{d\over dt}\mu_M^t |_{t=0}= \frac12 ({\rm tr}_Q  h) \mu_M=\frac12\langle g_Q,h\rangle \mu_M,
\end{align*}
where $\mu_M^t$ is the volume form  with respect to $g(t)$.
\end{lem}
\begin{proof}
Note that  for a nonsingular matrix $A$, 
\begin{align}\label{5-1}
{d\over dt} {\rm det}(A) = {\rm tr}[ {\rm det}(A) A^{-1} A'].
\end{align}
From (\ref{5-1}), we get
\begin{align}\label{5-2}
{d\over dt} \sqrt{{\rm det}(g_{ab}(t))}|_{t=0} = \frac12 ({\rm tr}_Q h) \sqrt{{\rm det}(g_{ab})}.
\end{align}
Since $\chi_{\mathcal F}$ is time independent, ${d\over dt}\mu_M^t = ({d\over dt}\nu^t)\wedge \chi_{\mathcal F}$ and 
 the transversal volume form $\nu^t$ is $\nu^t =\sqrt{{\rm det}(g_{ab}(t))} dy^1\wedge\cdots \wedge dy^q$.  Hence from (\ref{5-2})
\begin{align*}
{d\over dt}\mu_M^t|_{t=0} = ({d\over dt}\nu^t|_{t=0})\wedge \chi_{\mathcal F}
=\frac12 ({\rm tr}_Q h) \nu \wedge \chi_{\mathcal F} =\frac12 ({\rm tr}_Q h) \mu_M,
\end{align*}
 which completes the proof.
\end{proof}
 Then we have the following variation formula of the transversal metric. 
 \begin{thm} Let $\phi:(M,g,\mathcal F)\to (M',g',\mathcal F')$ be a  smooth foliated map.   On a compact domain $\Omega$ of $M$, we have 
 \begin{align*}
 {d\over dt} E_{B}(\phi,g(t),\Omega)|_{t=0} =\frac12\int_\Omega \langle\frac12|d_T\phi|^2 g_Q -\phi^* g_{Q'} , h\rangle \mu_M.
 \end{align*}
  \end{thm}
\begin{proof} By Lemma 5.1, we have
\begin{align*}
{d\over dt} E_{B}(\phi,g(t),\Omega)|_{t=0} = \frac12\int_\Omega( {d\over dt}|d_T\phi|^2|_{t=0})  \mu_M + \frac14\int_\Omega \langle |d_T\phi|^2 g_Q,h\rangle \mu_M.
\end{align*}
On the other hand,  since ${d\over dt} g^{ab} =-h^{ab} (=\sum_{c,d} g^{ac}g^{bd} h_{cd})$,  we have
\begin{align}\label{5-4}
{d\over dt}|d_T\phi|^2 |_{t=0}=  \sum_{a,b}{d\over dt} g^{ab}|_{t=0} g_{Q'}(d_T\phi({\partial\over \partial y^a}), d_T\phi({\partial\over \partial y^b}))=-\langle h, \phi^* g_{Q'}\rangle.
\end{align}
Hence the proof follows from (\ref{5-4}).
\end{proof}
The {\it transversal stress-energy tensor} of a foliated map $\phi :(M,g,\mathcal F)\to (M',g',\mathcal F')$ is given by  \cite{JU3}
\begin{align}\label{5-4}
S_T(\phi) =\frac12|d_T\phi|^2 g_Q -\phi^* g_{Q'}.
\end{align}
Then we have the following corollary.
\begin{cor} Let $\phi:(M,g,\mathcal F)\to (M',g',\mathcal F')$ be a nonconstant foliated smooth map with $M$ compact.  Then $\phi:M \to M'$ is an extremal of the transversal energy functional $E_{B}(\phi)$ with respect to variations of the transversal metric $g_Q$ if and only if $S_T(\phi)=0$. 
\end{cor}
If there exists a basic function $\lambda^2$ such that $\phi^* g_{Q'} =\lambda^2 g_Q$, then $\phi$ is called a {\it transversally weakly  conformal map}. In the case of $\lambda$ being nonzero constant, $\phi$ is called a {\it transversally homothetic} map.
Hence we have the following propositions.
\begin{prop}  Let $\phi:(M,g,\mathcal F)\to (M',g',\mathcal F')$ be a nonconstant foliated smooth map. Then
$S_{T}(\phi)=0$ if and only if  $q=2$ and $\phi$ is transversally weakly conformal.
\end{prop}
\begin{proof}  The proof is easy from (\ref{5-4}).
\end{proof}
\begin{prop} Let $\phi:(M,g,\mathcal F)\to (M',g',\mathcal F')$ be a transversally weakly conformal map and $q>2$. Then
$\phi$ is transversally homothetic if and only if $\phi$ satisfies the transversally conservation law.
\end{prop}
\begin{proof}  
Let $\phi$ be a transversally weakly conformal map, i.e., $\phi^* g_{Q'} =\lambda^2 g_Q$ for a basic function $\lambda$.  Fix $x$. Let $\{E_a\}$ be a local orthonormal basic frame on $Q$ such that $\nabla E_a=0$ at $x$. Then at $x$, we get
\begin{align*}
({\rm div}_\nabla S_T(\phi))(E_a)& = \sum_b (\nabla_{E_b}S_T(\phi))(E_b,E_a) \\
&=\sum_{b,c}\nabla_{E_b} \Big(\frac12  g_{Q'}( d_T\phi(E_c),d_T\phi(E_c)) \delta_{ab} - \phi^*g_{Q'}(E_b,E_a)\Big) \\
&=\sum_c \nabla_{E_a} (\frac12 \lambda^2 \delta_{cc} ) -\nabla_{E_a}\lambda^2\\ 
&=\nabla_{E_a}  ({q-2\over 2} \lambda^2).
\end{align*}
Hence we have
\begin{align*}
{\rm div}_\nabla S_T(\phi) = {q-2 \over 2} d_B (\lambda^2),
\end{align*}
which finishes the proof.
\end{proof}
Now, we put 
\begin{align}
S_{T,f}(\phi) = e^{-f} S_T(\phi),
\end{align}
called the {\it transversal stress $f$-energy tensor} of $\phi$.  
\begin{lem} Let $\phi:(M,g,\mathcal F,e^{-f}\nu)\to (M',g',\mathcal F')$ be a smooth foliated map. Then
\begin{align*}
({\rm div}_\nabla S_{T,f}(\phi) )(X) = - \langle \tau_{b,f}(\phi),d_T\phi(X)\rangle -\frac12 e^{-f}|d_T\phi|^2 d_B f (X)
\end{align*}
for any normal vector  $X\in\Gamma Q$.
\end{lem}
\begin{proof}  Fix $x$. Let $\{E_a\}$ be a local orthonormal basic frame on $Q$ such that $\nabla E_a=0$ at $x$. Then at $x$,
\begin{align*}
({\rm div}_\nabla S_{T,f}(\phi))(X) &= \sum_a (\nabla_{E_a} S_{T,f}(\phi))(E_a,X)\\ 
&=e^{-f}({\rm div}_\nabla S_T(\phi))(X) - e^{-f}S_T(\phi)(\nabla_{\rm tr} f,X)\\
&= -e^{-f} \langle \tau_b(\phi), d_T\phi(X)\rangle  -e^{-f}\Big(\frac12 |d_T\phi|^2 g_Q -\phi^* g_{Q'}\Big)(\nabla_{\rm tr} f,X)  \\
&=-\langle \tau_b(\phi) - d_T\phi(\nabla_{\rm tr}f), d_T\phi(X)\rangle e^{-f} - \frac12 e^{-f}|d_T\phi|^2 d_B f(X)\\
&=-\langle \tau_{b,f}(\phi),d_T\phi(X)\rangle -\frac12 e^{-f}|d_T\phi|^2 d_Bf (X),
\end{align*}
which finishes the proof. 
\end{proof}
\begin{rem}  Note that  ${\rm div}_\nabla S_T(\phi) (X) =\langle \tau_b(\phi),X\rangle$ for any normal vector field $X$. So a transversally harmonic map satisfies  the transverse conservation law, i.e, ${\rm div}_\nabla S_T (\phi)=0$ \cite{JU3},  but a transversally $f$-harmonic map does not satify the {\it transverse $f$-conservation law}, i.e., ${\rm div}_\nabla S_{T,f}(\phi)=0$ (cf. Lemma 5.6).
\end{rem}
From Lemma 5.6, we have the following Liouville type theorem.
\begin{prop} Let $f$ be a non-constant basic function. If a transversally $f$-harmonic map $\phi:(M,g,\mathcal F, e^{-f}\nu)\to (M',g',\mathcal F')$  satisfies the transverse $f$-conservation law, then $\phi$ is  transversally constant.
\end{prop}
Let $F:[0,\infty)\to [0,\infty)$ be a $C^2$-function such that $F'>0$ on $(0,\infty)$. The {\it transversally $F$-harmonic map} $\phi:M\to M'$ is a solution of  the Eular-Lagrange equation $\tau_{b,F}(\phi)=0$ \cite{CW}, where $\tau_{b,F}(\phi)$ is the transversal $F$-tension field given by
\begin{align}\label{5-6}
\tau_{b,F}(\phi)  = F' ({|d_T\phi|^2\over 2})\tau_b(\phi) + d_T\phi\Big(\nabla_{\rm tr} F'({|d_T\phi|^2\over 2})\Big).
\end{align}
When $F(s) =s$, the transversal $F$-tension field $\tau_{b,F}(\phi)$ is  the transversal tension field $\tau_b(\phi)$.
\begin{prop} A transversally $F$-harmonic map $\phi: (M,g,\mathcal F)\to (M',g',\mathcal F')$  without critical points  is a transversally $f$-harmonic map with $f= -\ln F'({|d_T\phi|^2\over 2})$.
\end{prop}
\begin{proof}  If $f= -\ln F'({|d_T\phi|^2\over 2})$ in (\ref{5-6}), then $\tau_{b,F}(\phi) =\tau_{b,f}(\phi)$.  So the proof is trivial.
\end{proof}
Now, we define the  {\it transversal $F$-stress energy} tensor $S_{T,F}(\phi)$ by
\begin{align}
S_{T,F}(\phi) = F({|d_T\phi|^2\over 2}) g_Q - F'({|d_T\phi|^2\over 2}) \phi^* g_{Q'}.
\end{align}
Then we have the following Lemma \cite{CW}.
\begin{lem}  Let $\phi: (M,g,\mathcal F)\to (M',g',\mathcal F')$ be a smooth foliated map. Then
\begin{align}
{\rm div}_\nabla S_{T,F}(\phi) = -\langle \tau_{b,F}(\phi), d_T\phi\rangle.
\end{align}
\end{lem}
\begin{proof} 
Fix $x$. Let $\{E_a\}$ be a local orthonormal basic frame  such that $\nabla E_a=0$ at $x$. Then at $x$,
\begin{align*}
({\rm div}_\nabla S_{T,F}(\phi))(E_a)&=\sum_b\nabla_{E_b}S_{T,F}(E_b,E_a)\\
 &=\sum_b \nabla_{E_b}\Big( F({|d_T\phi|^2\over 2})\delta_{ab} - F'({|d_T\phi|^2\over 2})\phi^*g_{Q'}(E_a,E_b)\Big)\\
 &=g_Q (\nabla_{\rm tr} F({|d_T\phi|^2\over 2}),E_a) -g_{Q'}( d_T\phi \Big(\nabla_{\rm tr} F'({|d_T\phi|^2\over 2})\Big),d_T\phi(E_a))\\
&-  g_{Q'} (F'({|d_T\phi|^2\over 2}) \tau_b(\phi),d_T\phi(E_a)) - \frac12 g_Q (F'({|d_T\phi|^2\over 2})\nabla_{\rm tr} {|d_T\phi|^2}, E_a).
\end{align*}
On the other hand, by the chain rule, we have
\begin{align*}
g_Q (\nabla_{\rm tr} F({|d_T\phi|^2\over 2}),E_a)&=E_a [F({|d_T\phi|^2\over 2})]\\
& = F'({|d_T\phi|^2\over 2}) E_a ({|d_T\phi|^2\over 2})\\
& = \frac12 g_Q (F'({|d_T\phi|^2\over 2})\nabla_{\rm tr}{|d_T\phi|^2},E_a).
\end{align*}
Hence from the above equations, we have
\begin{align*}
({\rm div}_\nabla S_{T,F}(\phi))(E_a)&=  -g_{Q'}( d_T\phi \Big(\nabla_{\rm tr} F'({|d_T\phi|^2\over 2})\Big) +F'({|d_T\phi|^2\over 2}) \tau_b(\phi),d_T\phi(E_a))\\
&=-  g_{Q'} (\tau_{b,F}(\phi), d_T\phi (E_a)),
\end{align*}
which finishes the proof.
\end{proof}
If $\phi:M\to M'$ satisfies ${\rm div}_\nabla  S_{T,F}(\phi) =0$, then we say $\phi$ satisfies the {\it transverse $F$-conservation law}.  Then we have the following. 
\begin{prop}  A transversally $F$-harmonic map satisfies the transverse $F$-conservation law. In particular,  if $f=-\ln F'({|d_T\phi|^2\over 2})$, then a transversally $f$-harmonic map satisfies the transverse $F$-conservation law, that is, ${\rm div}_\nabla S_{T,F}(\phi)=0$.
\end{prop}
\begin{proof} From Proposition 5.9 and Lemma 5.10, the proof follows.
\end{proof}

\end{document}